\documentclass{amsart}
\usepackage{amssymb}
\usepackage{pb-diagram}
\usepackage{color}
\usepackage[normalem]{ulem}

\newtheorem{theorem}{Theorem}[section]
\newtheorem{thm}[theorem]{Theorem}

\newtheorem{claim}[theorem]{Claim}
\newtheorem{fact}[theorem]{Fact}
\newtheorem{cor}[theorem]{Corollary}

\newtheorem{lemma}[theorem]{Lemma}

\newtheorem{question}[theorem]{Question}

\theoremstyle{definition}
\newtheorem{defn}[theorem]{Definition}

\newtheorem{example}[theorem]{Example}

\theoremstyle{remark}

\newtheorem{remark}[theorem]{Remark}

\newcommand{\WM}{\widetilde{\cal M}}

\newcommand{\la}{\langle}
\newcommand{\ra}{\rangle}

\newcommand{\CH}{{\cal H}}
\newcommand{\CU}{{\cal U}}
\newcommand{\sub}{\subseteq}

\newcommand{\dcl}{\operatorname{dcl}}

\newcommand{\scl}{\operatorname{scl}}

\newcommand{\dscl}{\operatorname{dscl}}

\newcommand{\ldim}{\dim}

\newcommand{\bb}[1]{\ensuremath{\mathbb{#1}}}

\newcommand{\cal}[1]{\ensuremath{\mathcal{#1}}}

\newcommand{\Lrarr}{\ensuremath{\Leftrightarrow}}

\newcommand{\res}{\ensuremath{\upharpoonright}}

\newcommand{\es}{\ensuremath{\emptyset}}

\newcommand{\sm}{\setminus}

\newcommand{\Z}{\mathbb{Z}}
\newcommand{\N}{\mathbb{N}}
\newcommand{\Q}{\mathbb{Q}}
\newcommand{\R}{\mathbb{R}}

\title[Characterizing o-minimal groups]
{Characterizing o-minimal groups in tame expansions of o-minimal structures}


\date{\today}

\begin{document}

\author {Pantelis  E. Eleftheriou}

\address{Department of Mathematics and Statistics, University of Konstanz, Box 216, 78457 Konstanz, Germany}

\email{panteleimon.eleftheriou@uni-konstanz.de}

\thanks{Research supported by an Independent Research Grant from the German Research Foundation (DFG) and a Zukunftskolleg Research Fellowship.}

\begin{abstract} 


Let  $\WM=\la \cal M, P\ra$ be an expansion of an o-minimal structure $\cal M$ by a dense set $P$, such that three tameness conditions hold. Among all definable groups in $\WM$, we characterize those that are definable in $\cal M$ as the ones whose dimension equals the dimension of their topological closure.
As an application, we obtain that if $P$ is independent, 
then every definable group in $\WM$ is already definable in $\cal M$.






\end{abstract}

\begin{abstract}   We study groups definable in tame expansions of o-minimal structures and give a dimension-theoretic characterization of those groups that are o-minimal. Let  $\cal N$ be an expansion of an o-minimal structure $\cal M$,  such that every open definable set is definable in \cal M, and such that \cal N admits a dimension function compatible with \cal M.  We prove that a definable group is definable in \cal M if and only it has maximal dimension; namely, its dimension equals the dimension of its topological closure. Our setting includes all known tame expansions $\cal N = \la \cal M, P\ra$, where $P\sub M$ is a dense set, such as dense pairs, Mann pairs and expansions by an independent set. As an application,  in the last case we obtain that every definable group is already definable in \cal M.

\end{abstract}

\begin{abstract}
In this paper we prove the first global results for groups definable in tame expansions of o-minimal structures. Let $\cal N$ be an expansion of an o-minimal structure \cal M that admits a  dimension function with the usual properties. The setting includes all known tame expansions $\cal N=\la\cal M, P\ra$, where $P$ is a dense set, such as dense pairs, Mann pairs, and expansions by an independent set. Our results include: (1) a Weil's group chunk theorem that guarantees that a definable group with an o-minimal group chunk is o-minimal, (2) a full characterization of those definable groups that are o-minimal as those definable groups that have maximal dimension; namely their dimension equals the dimension of their topological closure, (3) an application, that in the expansion $\cal N=\la \cal M, P\ra$, where $P$ is independent, every definable group is o-minimal.
\end{abstract}

\begin{abstract}
We establish the first global results for groups definable in tame expansions of o-minimal structures. Let $\cal N$ be an expansion of an o-minimal structure \cal M that admits a good  dimension theory. The setting includes dense pairs of o-minimal structures, expansions of \cal M by a Mann group, or by a subgroup of an elliptic curve, or a dense independent set. We prove: (1) a Weil's group chunk theorem that guarantees a definable group with an o-minimal group chunk is o-minimal, (2) a full characterization of those definable groups that are o-minimal as those groups that have maximal dimension; namely their dimension equals the dimension of their topological closure, (3) if $\cal N$ expands \cal M by a dense independent set, then  every definable group is o-minimal.


\end{abstract}

\maketitle

\section{Introduction}

Definable groups have been at the core of model theory for at least a period of three decades, largely because of their prominent role in important applications of the subject, such as Hrushovski's proof of the function field Mordell-Lang conjecture in all characteristics (\cite{hr}). Examples include  algebraic groups (which are definable in algebraically closed fields) and  real Lie groups (which are definable in o-minimal structures). Groups definable in o-minimal structures are well-understood, with arguably the most influential work in the area being the solution of Pillay's conjecture over a field (\cite{hpp}), which brought to light new tools in  theories with NIP.
On the other hand, groups definable in tame expansions of o-minimal structures have only  been  studied locally (\cite{egh}). In this paper we prove the first global results, 
whose gist is that one can recover a group definable in the o-minimal reduct from an arbitrary definable group using only dimension-theoretic data.


O-minimal structures were introduced and first studied by van den Dries \cite{vdd-tarski} and Knight-Pillay-Steinhorn \cite{kps, ps} and have since provided a rigid framework to study real algebraic and analytic geometry. They have enjoyed a wide spectrum of applications reaching out even to number theory and Diophantine geometry (such as in Pila's solution of certain cases of the Andr\'e-Oort  Conjecture \cite{pila}).
Tame expansions of o-minimal structures have been developed as a context that escapes  the o-minimal, locally finite setting, yet preserves the tame geometric behavior on the class of all definable sets.
 An important category of such structures are those where every open definable set is already definable in the o-minimal reduct.
The primary example is that of the real field expanded by the subfield of real algebraic numbers, studied by A. Robinson in his classical paper \cite{rob}, where the decidability of its theory was proven.
Forty years later, van den Dries  \cite{vdd-dense} extended Robinson's results to arbitrary dense pairs of o-minimal structures, and a stream of further developments in the subject followed (\cite{bz, beg, bh,dms1,dms2,dg, gh,ms}). Besides dense pairs, examples of structures in this category  now include pairs of the form $\la \cal M, P\ra$, where $\cal M$ is an o-minimal expansion of an ordered group, and $P$ is a dense multiplicative subgroup with the Mann property, or a dense subgroup of the unit circle or of an elliptic curve, or it is  a dense independent set. Moreover, a cone decomposition theorem and the associated dimension function have been developed in a  general setting that includes the above pairs  (\cite{egh}), extending the known cell decomposition theorem from o-minimal structures and the usual o-minimal dimension.
The setting of the current paper is even more general, as only the existence of a good dimension theory is assumed. Moreover, the o-minimal reduct does not need to expand an ordered group.  Our main theorem (Theorem \ref{main} below) strikingly reflects the underlying idea of this category
(that open definable sets are definable in the o-minimal reduct) at the level of definable groups.  Let us introduce some terminology and explain its concept.



 Throughout this paper, $\cal M$    and $\cal N$ denote two first-order structures, with $\cal N$ expanding $\cal M$. We denote by $\cal L$ the language of \cal M  and by $\dcl$ its usual definable closure.  By `$\cal L$-definable' or  `definable in \cal M' we mean definable in \cal M with parameters. By `definable'  or `definable in \cal N' we mean definable in \cal N with parameters. With the exception of Section \ref{sec-chunk}, $\cal M$ is o-minimal, and every open definable set is $\cal L$-definable.

 A \emph{dimension function compatible with \cal M} is a map $\dim$ from the class of all definable sets  to $\{-\infty\}\cup \N$ that satisfies the following properties: for all definable $X,Y\sub M^n$,  and $a\in M$,
\begin{enumerate}
\item[\textbf{(D1)}]  $\dim\{a\}=1$, $\dim M=1$, and $\dim X=-\infty$ if and only if $X=\es$
\item[\textbf{(D2)}]  $\dim (X\cup Y)=\max \{\dim X, \dim Y\}$
\item[\textbf{(D3)}]  if $\{X_t\}_{t\in I}$ is a disjoint definable family of sets, then
\begin{enumerate}
\item for $d=\{-\infty\}\cup \N$, the set $I_d=\{t\in I: \dim X_t=d\}$ is definable, and
\item  if every $X_t$ has dimension $k$, then
$$\dim \bigcup_{t\in I} X_t=\dim I + k$$
\end{enumerate}
\item[\textbf{(D4)}]  if $f:X\to Y$ is a definable bijection, then $\dim X=\dim Y$
  \item[\textbf{(D5)}]  the dimension of every  \cal L-definable set $X$ coincides with its usual o-minimal dimension
  \item[\textbf{(D6)}]   every definable map $f:M^n\to M$ agrees with an $\cal L$-definable map $F:M^n\to M$ outside a definable set of  dimension $<n$. 
\end{enumerate}
It follows from \textbf{(D2)} that $\dim$ is monotone, and from \textbf{(D1)-(D4)} that it is   a dimension function also in the sense of  \cite{vdd-dim}.


In the rest of this paper, and unless stated otherwise, we assume that \cal N admits a dimension function $\dim$ compatible with \cal M. In the aforementioned pairs $\la \cal M, P\ra$ the existence of such $\dim$ was established in \cite{egh} (details are postponed until Section \ref{sec-egh}). Moreover, in that context, a definable set $X$ was shown to have dimension $0$ if and only if $X$ is internal to $P$ if and only if no open interval is internal to $X$ (where internality is the usual notion from geometric stability theory). Such a set $X$ was called  `small'. In particular, $P$ is small. Likewise here, let us call a definable set \emph{small} if it has dimension $0$, and \emph{large}, otherwise. Observe that by monotonicity of $\dim$, the dimension of a definable set $X$ is bounded by the (o-minimal) dimension of its topological closure $cl(X)$. We call a definable set $X$ \emph{strongly large} if
$$\dim X=\dim cl(X).$$
We call a definable group strongly large if its domain is  strongly large. Every infinite small set is not strongly large. Every $\cal L$-definable set is strongly large. 
The main intuition is that, conversely, strongly large sets must behave like $\cal L$-definable sets. Our main theorem establishes this intuition at the level of definable groups.

\begin{thm}\label{main} 
Let \cal N be an expansion of an o-minimal structure \cal M such that
\begin{itemize}
  \item[(a)] every open definable set is definable in \cal M, and
\item[(b)] \cal N admits a dimension function compatible with \cal M.
\end{itemize}
Let $G$ be a definable group. Then $G$ is definably isomorphic to a group definable in \cal M if and only if it is definably isomorphic to a strongly large group. 
\end{thm}

In the rest of this introduction, we discuss the content of the above theorem, state some consequences, and illustrate the main difficulties of its proof.

Theorem \ref{main} is in the spirit of a classical theme in model theory; namely,  to recover a mathematical object in some restricted language given data of possibly different nature. For example, the influential Pila-Wilkie theorem (\cite{pw}) recovers a semialgebraic subset of a set $X$  definable in an arbitrary o-minimal structure given a number-theoretic condition on $X$. In our case, we recover an o-minimal group from a group $G$  definable in an expansion of an o-minimal structure given a dimension-theoretic condition on the domain of $G$.

We   next point out the need for including definable isomorphisms in the conclusion of Theorem \ref{main}.
Suppose $\cal N=\la  \overline \R, P\ra$ is the expansion of the real field $\overline \R$ by the field $P$ of algebraic numbers. Let $K=\la \R, +\ra$
and consider the definable bijection $f:P+\pi P\to P^2$, with $f(x+\pi y)=(x,y)$. Let $G$  be the disjoint union of $K\sm (P+\pi P)$ and $P^2$, equipped with the group structure induced from $K$ via the identity map on the first part and via $f$ on the second. Then $G$ is not strongly large, 
since $\dim G=1$ and $\dim cl(G)=2$, but it is definably isomorphic to the $\cal L$-definable group $K$.

Theorem \ref{main} puts a constraint on the existence of new definable groups, which has already been the theme of previous research, such as in \cite{bv1} and \cite{bv2}. Let $\cal N=\la \cal M, P\ra$ denote one of the aforementioned pairs from \cite{egh}. As $P$ itself can be a definable group, there are new small definable groups (and a study for those has recently been initiated in \cite{bm}). Of course, there are also new large definable groups, such as the direct product $P\times M$, but as a consequence of the above theorem, there are no new large definable groups contained in $M$, or in any $1$-dimensional \cal L-definable set. A special case of this statement was proven in \cite[Lemma 7.3]{bv2}; namely, that there are no new definable \emph{subgroups} of $1$-dimensional \cal L-definable groups. Returning to arbitrary dimensions, observe that, by (\textbf{D3)}, $\dim (P \times M)=1$,  whereas $\dim cl(P\times M)=M^2=2$. Hence $P\times M$ is not strongly large and Theorem \ref{main} does not apply to it. What Theorem \ref{main} implies is that there are no new definable groups of dimension $n$ contained in $M^n$, or in any $n$-dimensional $\cal L$-definable set.


In the example $\cal N=\la \cal M, P\ra$, where  $P$ is a dense $\dcl$-independent set, our work implies that there are no new definable groups at all. This pair  recently received special attention in \cite{dms2} and even triggered new model-theoretic work at the general level of `$H$-structures' \cite{bv1}. The basic intuition is that a $\dcl$-independent set $P$ is at the other end of being a group, since there are no $\cal L$-definable relations between its elements. It is then easy to see that there are no new small definable groups, as those would have to be internal to $P$. Using the cone decomposition theorem from \cite{egh}, we prove that every definable group is definably isomorphic to a strongly large group (Theorem \ref{indsl}). Combined with Theorem \ref{main}, we obtain the following application.




\begin{theorem}\label{main2} Let $\cal M$ be an o-minimal expansion of an ordered group, and $\cal N=\la \cal M, P\ra$ an expansion of $\cal M$ by a dense $\dcl$-independent set $P$. Then every group definable in \cal N is definably isomorphic to a group definable in \cal M.
\end{theorem}

\noindent Again, a special case of this statement was previously proved, in \cite[Proposition 6.4]{bv1}; namely, that every definable \emph{subgroup} of $\la M^n, +\ra$ is \cal L-definable.  As a parallel note, Theorem \ref{main2} applies also to interpretable groups, as those are definably isomorphic to definable ones (by elimination of imaginaries \cite{dms2}). Elimination of imaginaries is known to fail in the general setting of Theorem \ref{main} (\cite{dms1}).




Finally, let us point out that Theorem \ref{main} establishes a conjecture for definable groups stated in \cite{el-bsl} and reformulated in \cite{egh}, for the case of strongly large groups. The conjecture stated that if $G$ is a definable group of dimension $k$, then there is a $\bigvee$-definable cover $\cal U$ of $G$, and a short exact sequence
$$
\begin{diagram}
\node{0}\arrow{e}\node{\CH} \arrow{e}\node{\cal U}
 \arrow{e}\node{K}
\arrow{e} \node{0}
\end{diagram} $$
where $\CH$  is $\bigvee$-definable  in \cal M, with $\dim \CH=k$, and $K$ is  definable and small.   Theorem \ref{main} implies the conjecture for $G$ strongly large, with $\CU=\CH=G$ and $K$ trivial. Namely, it answers \cite[Question 7.11]{egh}) affirmatively. It is worth noting that the above conjecture was inspired by an analogous theorem in a different setting (\cite{ep-sel2}), namely that of \emph{semi-bounded} o-minimal structures (see also Remark \ref{rmk-sbd} below).

There is a number of different  settings where at least the methods of this paper could  apply, such as that of $H$-structures, whereas a new direction is proposed in Section \ref{sec-future}. Let us now  proceed to describe the strategy of our proof.


$ $\\
\noindent\textbf{Sketch of the proof of Theorem \ref{main}}.
We  illustrate the main difficulties in proving the right-to-left direction of Theorem \ref{main}. Given a strongly large group $G$, we need to recover a suitable $\cal L$-definable domain $X$ and an $\cal L$-definable map $F:X^2\to X$ that can play the role of an $\cal L$-definable group  
definably  isomorphic to $G$. The construction is carried out in  Section \ref{sec-proof} through a series of six steps. 
 We next list those steps and verify them afterwards in an example.
 Let $G$ be a strongly large group with $G\sub M^n$ and $\dim G=k$. \\

\noindent (I) Recover an $\cal L$-definable map $F:V^2\to M^n$, with $\dim (G\triangle V)<k$, that agrees with $\cdot$ on a definable set $C\sub (V\cap G)^2$ with $\dim (V^2\sm C)<2k$ (equivalently, $\dim (G^2\sm C)<2k$).
This is possible because $G$ is strongly large. Indeed, in Lemma \ref{slF} we prove a generalization of \textbf{(D6)} for maps  $f$ with strongly large domain, which we can then apply to $\cdot : G^2\to G$ (Lemma \ref{firststep}). Of course, $C$ may not be $\cal L$-definable. If it were, we could directly skip to Step VI.\\


\noindent (II) Prove that $F$ satisfies `group-like' properties on an $\cal L$-definable subset $U$ of $V$ with  $\dim(V^2\sm U) <2k$, such as injectivity in each coordinate, and associativity. This is done using (I) and the fact that $\cdot$ is a group operation. Moreover,  $F(U)\sub V$.\\



\noindent(III) Show that  $X=F(\Gamma^2\cap U)$ is an $\cal L$-definable set with $\dim(V\sm X)<k$. This is done using earlier work from Section \ref{sec-localLdef} for extracting $\cal L$-definable sets.\\



\noindent (IV) Construct a suitable definable embedding $h:X \to G$. This is the heart of the whole proof. We first prove that
\begin{align}
(*)\,\,\,\,\,\,\,\,\, &\text{for every $t\in X$, there is unique $r\in G$, such that} \notag\\
 &\text{the set $\{x\in V\cap G : F(t,x)=r \cdot x\}$ is of co-dimension $<k$ in $V$ (or $G$),} \notag
\end{align}
and then  define $h(t)=r$ via (*).\\


\noindent
(V) Show that there is an $\cal L$-definable set $Z\sub X^2$ with $\dim (X^2\sm Z)<2k$  and $F(Z)\sub X$, such that for every $(t,x)\in Z$,
$$ h(F(t, x))=h(t)\cdot h(x).$$
The proof combines all information for $F$, $X$ and $h$ from Steps (II)-(IV).\\

\noindent (VI) Apply a  group chunk theorem to the quadruple $(Z, X, F, h)$ to conclude that $G$ is definably isomorphic to an $\cal L$-definable group. This group chunk theorem is proved in Section \ref{sec-chunk} in a higher generality, where $\cal M$ and $\cal N$ are arbitrary structures satisfying only some of the dimension axioms \textbf{(D1)-(D6)}.



\begin{example}\label{exa-ms}
Suppose $\cal N=\la \cal M, P\ra$ is an expansion of an ordered group  \cal M by a dense elementary substructure $P$. Let $K=\la M, +\ra$ and denote $\Gamma= M\sm P$. Consider the definable bijection $h:M \to M$ given by
$$h(x)=\begin{cases}
  x, & \text{ if $x\in \Gamma$},\\
  x+1, &\text{ if $x\in P$}.
\end{cases}
$$
and let $G=\la M, \cdot\ra$ be the induced group structure on $M$. Namely, if we write $F$ for the map $+$, then for every $(t,x)\in M^2$,
\begin{equation}
  h(F(t, x))=h(t)\cdot h(x).\notag
\end{equation}
Then $G$ is strongly large (even with $\cal L$-definable domain), and it is definably isomorphic to the $\cal L$-definable group $K$ via $h$. We would like to recover $h$ in an abstract way. This is done in Step (IV) below. However, we illustrate all steps from the general procedure.  Let $V=M$.

\smallskip
\noindent (I) For every $t, x\in M$, we have
$$t \cdot x=t+x \,\Lrarr\, t,x,t+x\in \Gamma.$$ So if we let for every $t\in \Gamma$, $C_t=\Gamma \cap (\Gamma - t)$, then $\cdot$ agrees with $F$ exactly on the set
$$C=\bigcup_{t\in \Gamma} \{t\}\times C_t.$$
Since each $C_t$ is co-small in $M$, it follows from \textbf{(D3)} that $\dim (M^2\sm C)<2$.

\smallskip
\noindent (II)  Let $U=V^2$.

\smallskip
\noindent (III) We have  $X=F(\Gamma^2)=\Gamma +\Gamma=M$.

\smallskip
\noindent (IV) We prove that here (*) actually yields exactly $h$. Namely,
\begin{align} h(t)= \text{ the unique $r\in G$, such that the set $\{x\in M : F(t,x)=r \cdot x\}$ is  co-small.} \notag
\end{align}
To see this, one could perform a direct computation, or argue as follows (also in preparation for the sort of arguments that take place in general). Consider the following equalities:
$$F(t, x)=h(F(t, x))=h(t) \cdot h(x)=h(t) \cdot x.$$
The first equality holds for all those $x\in \Gamma$ such that $F(t, x)\in \Gamma$, and hence, by injectivity of $F$ in the second coordinate, for co-small many $x$. The second equality holds for every $x\in M$. The third equality holds for all $x\in \Gamma$, again, co-small many. All together, $F(t,x)=r \cdot x$ holds for co-small many $x$.
Moreover, there can only be one such $r$ satisfying (*), because otherwise we would have two co-small disjoint subsets of $M$, a contradiction. We have thus shown that (*) yields $h$.

\smallskip
\noindent (V) Let $Z=X^2$.

\smallskip
\noindent (VI) The group chunk theorem here is not needed, as we actually have $K=\la X, F\ra$, and $h:K\to G$ is the desirable definable isomorphism.
\end{example}

\begin{remark}
The idea of recovering a group operation via (*) is a recast of a similar idea in \cite{mst}. In Section 1.3 of that reference, the authors recover a definable  isomorphism $h$ between $\la M, <, +\ra$ and an ordered group $\la M, <, \oplus\ra$, satisfying  additional properties, as follows:
$$t\mapsto \lim_{x\to \infty} [(t+x)\ominus x].$$
In Example \ref{exa-ms}, instead of letting $h(t)$ to be such a limit as $x\to \infty$, we require the equation $t+x=h(t) \cdot x$ to hold for co-small many $x$.
\end{remark}

\begin{remark}
The first attempt to recover an $\cal L$-definable domain in Step (III) of the general procedure would probably be to take $X=F(C)$. Besides, this set is always contained in $G$, and hence the rest of the analysis (IV)-(VI) could be simplified. But $F(C)$ need not be $\cal L$-definable; indeed, in Example \ref{exa-ms}, $F(C)=\Gamma$. On the other hand, the set $F(\Gamma^2\cap U)$ we construct is always $\cal L$-definable, but it need not be contained in $G$ (it would be if $\Gamma^2\sub C$). This can be seen in a similar example, where instead of permuting $P$ via $h$, we move it away from $M$, say to $P\times \{1\}$, via $h(x)=(x,1)$. Then $G$ contains no infinite $\cal L$-definable set (and if we let $F=+$ and $\Gamma=M\sm P$, then again  $F(\Gamma^2)= M$.)
\end{remark}

 $ $\\
\noindent\textbf{Acknowledgements.} I wish to thank Ya'acov Peterzil for pointing out the relevant literature and discussing the proof of the group chunk theorem in Section \ref{sec-chunk}. The relevant discussions took place during the trimester in model theory, combinatorics and valued fields, 2018, at the Institut Henri Poincar\'e. I also thank Alfred Dolich and Philipp Hieronymi for suggesting that there should be no new definable groups in the setting of Theorem \ref{main2}. Finally, I thank Chris Miller for his helpful feedback on an earlier version of this paper.

\section{Preliminaries}



In this section, we fix some notation, prove some basic facts, analyze strongly large sets, and show how the pairs $\la \cal M, P\ra$ from \cite{egh} fit  to the current setting.


\subsection{Notation}
The topological closure of a set $X\sub \R^n$ is denoted by $cl(X).$ If $X, Z\sub \R^n$, we call $X$ dense in $Z$ if $Z\sub cl(X\cap Z)$. We call $X$ co-dense in $Z$ if $Z\sm X$ is dense in $Z$. Given a set $X\subseteq \R^m \times \R^n$ and $a\in \R^m$, we write $X_a$ for
\[
\{ b \in \R^n \ : \ (a, b) \in X\}.
\]
We write $\pi:M^n\to M^{n-1}$ for the projection onto the first $n$ coordinates, unless stated otherwise. If $X, Y \subseteq \R$, we sometimes write $XY$ for $X\cup Y$.
A tuple of elements is denoted just by one element, and we write $b\sub B$ if $b$ is a tuple with coordinates from $B$.  Our use of the notions of a $k$-cell, open and closed box are standard. 
We say that a set $X$ has co-dimension $<k$ in $Y$, if $\dim (Y\sm X)<k$. We write $X\triangle Y=(X\sm Y)\cup (Y\sm X)$.
We say that a map $f: A\sub M^n\times M^n\to M^k$ is injective in each coordinate if, for every $a, b\in M^n$, $f(a, -)$ and $f(-, b)$ are injective. We identify $M^n\times M^n$ with $M^{2n}$. By an embedding we mean an injective map.




\subsection{Basic facts}

\begin{fact}\label{denseldim} Let $V\sub M^n$ be a $k$-cell and $X\sub V$ a definable set with $\ldim(V\sm X)<k$. Then $X$ is dense in $V$.
\end{fact}
\begin{proof} Since $V$ is a $k$-cell, a relatively open subset $B$ of $V$ has dimension $k$.   Since $\dim (B\sm X)<k$, we obtain  $\dim (B\cap X)=k$. In particular, $B\cap X\ne \es$.
\end{proof}

\begin{fact}\label{denseZ0}  Let $Z\sub M^n$ be an $\cal L$-definable set of dimension $k$, and $Z'\sub Z$ a $k$-cell. If a definable set is dense in $Z$, then so is it in $Z'$.
\end{fact}
\begin{proof}
\cite[Lemma 2.6]{el-pw}.
\end{proof}


The following lemma will be used in the proof of Lemma \ref{Finj}. It generalizes \cite[Proposition 4.19]{egh}.

\begin{lemma}\label{inj-lemma}
Let  $F:X\sub M^n\to M^m$ be an $\cal L$-definable map, and $D\sub X$ a definable set with $\ldim(X\sm D)< \dim X$. Assume  that $F_{\res D}$ is injective. Then there is an $\cal L$-definable set $Y\sub X$ such that $\dim (X\sm Y)<\dim X$ and $F_{\res Y}$ is  injective. 
\end{lemma}
\begin{proof} Assume $\dim X=k$. 
 Denote
$$T=\{a\in X: F^{-1}(F(a)) \text{ is finite}\}.$$
We claim that $\dim(X\sm T)<k$. Assume not, and let $C\sub X\sm T$ be a $k$-cell. Then $\dim (C\cap D)=k$. Now, by \cite[Fact 2.9]{egh}, $F(C)$ has dimension $s<k$. In particular, $F(C)$ is in \cal L-definable bijection with a subset of $M^s$. Hence $C\cap D$ is in definable bijection with a subset of $M^s$, contradicting \textbf{(D4)}.

Now, by uniform finiteness in o-minimal structures, one can easily find
\begin{itemize}

\item an $\cal L$-definable set $B\sub T$ of dimension $k$, such that $F_{\res B}$ is injective, and
\item an $\cal L$-definable map $f:B\to T\sm B$, such that for every $x\in B$,
$$F(x)=F(f(x)).$$
\end{itemize}
Observe that then $f$ is injective, since if $x,y\in B$ and $f(x)=f(y)$, then $F(x)=F(f(x))=F(f(y))=F(y)$ which implies $x=y$. 
Moreover, since $F_{\res D}$ is injective,
$$f(B\cap D)\sub f(B)\sm D.$$
But 
$\ldim(B\cap D)=k$, and hence by injectivity of $f$, the set on the right also has  dimension $k$, contradicting $\ldim(X\sm D)<k$.
\end{proof}

\begin{question} In  Lemma \ref{inj-lemma}, can  it moreover be $D\sub Y$?
\end{question}

\subsection{Strongly large sets}\label{sec-sl}
Here we prove some statements about strongly large sets. We also introduce the notion of a \emph{full} set.  The first lemma extends property \textbf{(D6)} to functions $f$ whose domain is any strongly large set, instead of just $M^n$.



\begin{lemma} \label{slF}
Let $X\sub M^n$ be strongly large of dimension $k$. Then every definable map $f:X\to M$ agrees with an $\cal L$-definable map $F:M^n\to M$ outside a definable set $S$ of dimension $<k$.
\end{lemma}
\begin{proof}
By working with the coordinate functions of $f$, we may assume that $m=1$. Indeed, if we find a suitable set $S_i$ for the $i$-th coordinate $f_i$, then $S=\bigcup_i S_i$ works for $f$, by \textbf{(D2)}.

We may assume $X\sub M^k$. Indeed, $cl(X)$ is a finite union of cells. If $C$ is one of the cells and $\ldim(C\cap X)<k$,  we can disregard it. Otherwise, $\ldim(C\cap X)=k=\ldim cl(C\cap X)$, and hence it is enough to work with one of these. After projecting $C$ onto suitable coordinates, we may assume that $X\sub M^k$.

Define $H:M^k\to M$  as $H(x)=f(x)$, if $x\in X$, and $0$, otherwise. This map $H$ is definable, and hence, by \textbf{(D6)}, it agrees with an $\cal L$-definable map $F:M^k\to M$ outside a set $S$ of dimension $<k$. Then $f$ agrees with $F$ outside $S\cap X$.
\end{proof}

The above lemma  supports the intuition that strongly large sets behave like $\cal L$-definable sets. We strengthen the notion of being strongly large as follows.


\begin{defn}
A definable set $X$ is called \emph{full} if $\dim(cl(X)\sm X)<\dim X$.
\end{defn}

By \textbf{(D5)}, every $\cal L$-definable set is full. By \textbf{(D2)}, a full set is strongly large. The converse is not true; for example, let $X$ be the disjoint union of an open interval and an infinite small set.
Some natural examples of full sets come from the setting of \cite{egh}, see Fact \ref{exa-hJ} below.


In Section \ref{sec-proof}, we will use the following consequence of Lemma \ref{slF}.

\begin{cor}\label{slfull} Every strongly large set is a union of a full set and a set of smaller dimension.
\end{cor}
\begin{proof}
Let $Y\sub M^n$ be a strongly large set of dimension $k$. As in the proof of Lemma \ref{slF}, we may assume that $n=k$. Let $f: M^n\to M$ be a characteristic function for $Y$; namely, fix two elements $0, 1\in M$ and let $f(x)=1$, if $x\in Y$, and $f(x)=0$, otherwise.  By \textbf{(D6)}, $f$ agrees with an $\cal L$-definable map $F:M^n\to M$ outside a definable set $S$ of dimension $<n$. Let $C=\{x\in M^n :F(x)=1\}$. Since $f$ and $F$ agree on $C\sm S$, this means that $C\sm S\sub Y$. Since
$$cl(C\sm S)\sm (C\sm S)\sub (cl(C)\sm C)\cup S,$$
and the latter set has dimension $<n$, we obtain that $C\sm S$ is full. Since also
$$Y=(C\sm S) \cup (Y\cap S),$$
we are done.
\end{proof}


The above conclusion may fail if we do not assume that the given set is strongly large. For example,  consider any infinite small set. Also, in Corollary \ref{slfull2} below, we prove a partial converse of the above corollary for a pair $\cal N=\la \cal M, P\ra$ where $P$ is a dense $\dcl$-independent set. We do not know whether that converse is true in general.


In Section \ref{sec-indep}, we will also need the following.

\begin{lemma}\label{fullunion} A finite union of full sets is full.
\end{lemma}
\begin{proof} Let $X=X_1\cup \dots \cup X_m$, where each $X_i$ is a full set. By o-minimality, the union of the closures of finitely many $\cal L$-definable sets equals the closure of their unions. It follows easily that $cl(X)=\bigcup_i cl(X_i)$.
Therefore,
$$cl(X)\sm X=\left(\bigcup_i cl(X_i) \right)\sm \left( \bigcup_i X_i\right)= \bigcup_i \left(cl(X_i)\sm \bigcup_i X_i\right)\sub \bigcup_i (cl(X_i)\sm X_i),$$
and hence $\ldim(cl(X)\sm X)\le \max_i \ldim (cl(X_i)\sm X_i)<\dim X_i \le \dim X$.
\end{proof}


\subsection{The setting of \cite{egh}}\label{sec-egh} In \cite{egh} we studied pairs $\cal N=\la \cal M, P\ra$, where $P\sub M^n$, such that three tameness conditions hold. Following \cite{dg}, let us call a definable set $X\sub M^n$ \emph{large} if there is an $\cal L$-definable map $f:M^{nk}\to M$ such that $f(X^k)$ contains an open interval. Otherwise, it is called \emph{small}. The three conditions in \cite{egh} (see Section 2 there for more details) are: (I) $P$ is small, (II) $Th(\cal N)$ is near-model complete, and (III) every open definable  set is \cal L-definable. In \cite[Section 2.2]{egh}, the following examples were shown to fall into this category: (a) dense pairs, (b) expansions of the real field by a multiplicative subgroup with the Mann property, or by a  dense subgroup of the unit circle or of an elliptic curve, (c)  expansions by a dense independent set.

For  the rest of this section, let $\cal N=\la M, P\ra$ satisfy  conditions (I)-(III) above. In \cite{egh},  a suitable notion of dimension was introduced,
which we describe next.

\begin{defn}[\cite{egh}]\label{def-supercone}
A \emph{supercone} $J\sub M^k$, $k\ge 0$, and its \emph{shell} $sh(J)$ are defined recursively as follows:
\begin{itemize}
\item $M^{0}=\{0\}$ is a supercone, and $sh(M^{0})=M^{0}$.

\item A definable set $J\sub M^{n+1}$ is a supercone if $\pi(J)\sub M^n$ is a supercone and there are  $\cal L$-definable continuous maps $h_1, h_2: sh(\pi(J))\to M\cup \{\pm\infty\}$ with $h_1<h_2$, such that for every $a\in \pi(J)$, $J_a$ is contained in $(h_1(a), h_2(a))$ and it is co-small in it. We let $sh(J)=(h_1, h_2)_{sh(\pi(J))}$.
\end{itemize}
\end{defn}

Note that, $sh(J)$ is the unique open cell in $M^k$ such that $cl(sh(J))=cl(J)$.

\begin{defn}[Large dimension \cite{egh}]\label{def-large}
Let $X\sub M^n$ be definable. If $X\ne \emptyset$, the \emph{large dimension} of $X$ is the maximum $k\in \bb N$ such that $X$ contains a set of the form $f(J)$, where $J\sub M^k$ is a supercone and $f:sh(J)\to M^n$ is an $\cal L$-definable continuous injective map. The large dimension of the empty set is defined to be $-\infty$. 
\end{defn}

The large dimension was used in \cite{egh} to prove a cone decomposition theorem for all definable sets, in analogy with the cell decomposition theorem known for o-minimal structures. A consequence of this theorem was that the large dimension satisfies all properties \textbf{(D1)-(D6)} of the current paper. More precisely, these properties are established in \cite[Corollaries 5.3 and 5.5, Theorem 5.7, and Lemma 6.11]{egh}. Moreover, a definable set has large dimension zero if and only if it is small.
We will make use of the cone decomposition theorem in Section \ref{sec-indep}, and we delay it until then.
For now, let us  point out some basic facts, again to be used in Section \ref{sec-indep}, but can be stated under our general assumptions.

In what follows, the dimension $\dim$ denotes   the  large dimension.

\begin{fact}\label{fact-indt}
Let $X\sub M^k$ be a definable set of dimension $k$, and $A_0\sub M$ a finite set. Then there is $t\in X$ which is $\dcl$-independent over $A_0 P$.
\end{fact}
\begin{proof} By \cite[Theorem 5.7(1)]{egh},  $X$ contains a supercone, and hence we may assume that $X$ is a supercone. Consider the operator $\scl$ that maps $A\sub M$ to $\scl(A)=\dcl(AP)$. By \cite[Section 6]{egh}, $\scl$ defines a pregeometry and the corresponding $\scl$-dimension for definable sets agrees with $\dim$. It is then easy to see from the definition of supercones,  by induction on $k$, that there is $t\in X$ which is $\scl$-independent over $A_0$, and hence $\dcl$-independent over $A_0 P$, as required.
\end{proof}

In the next fact, we draw a connection to the full sets from the last subsection.

\begin{fact}\label{exa-hJ} let $f(J)$ be as in Definition \ref{def-large}. Namely, $J\sub M^k$ is a supercone and $f:sh(J)\to M^n$ is an $\cal L$-definable continuous injective map. Then $f(J)$ is a full set.
\end{fact}
\begin{proof}
We first note that $J$ is a full set, by \cite[Corollary 4.28]{egh}.  Now,  let $V=sh(J)$. Observe that
$$cl(f(J))\sm f(J)\sub cl(f(V))\sm f(J)=(cl(f(V))\sm f(V))\cup (f(V)\sm f(J)).$$
Since $f(V)$ is $\cal L$-definable, the first part of the last union has dimension $<k$. Since also $f$ is continuous and injective,  the set
$$f(V)\sm f(J) \sub f(V\sm J)$$
also has dimension $\dim (V\sm J)\le \dim (cl(J)\sm J)<k$, as needed.
\end{proof}





\section{Local $\cal L$-definability}\label{sec-localLdef}

This section contains a key result (Corollary \ref{coordinate inj3}) which will be used in the proof of Lemma \ref{F(V)} below in order to extract an $\cal L$-definable set from some given data. At first, one  recovers only a `locally \cal L-definable' set, which we prove that it is in fact \cal L-definable (Lemma \ref{Ldef}).

\subsection{Preliminaries on local $\cal L$-definability}

\begin{defn}
Let $V\sub M^n$ be a definable set, and $x\in V$. We call $V$ \emph{locally $\cal L$-definable at $x$} if there is an open box $B\sub M^n$ containing $x$ such that $B\cap V$ is $\cal L$-definable. We call $V$ \emph{locally \cal L-definable} if it is locally \cal L-definable at every point.
\end{defn}


The following fact follows easily from the definition.

\begin{fact}\label{locLdeffact}
Suppose $V\sub M^n$ is a $k$-cell and $X\sub V$ a definable set. Then $X$ is locally $\cal L$-definable if and only if  for every $x\in X$, there is a $k$-cell $B\sub X$ containing $x$.
\end{fact}

Of course, an $\cal L$-definable set is locally $\cal L$-definable. We will also prove the converse.

\begin{lemma}\label{VcapC}
Suppose $V\sub M^n$ is locally $\cal L$-definable and $X\sub M^n$ is $\cal L$-definable. Then $V\cap X$ is locally $\cal L$-definable.
\end{lemma}
\begin{proof}
If $B\sub M^n$ is an open box and $B\cap V$ is $\cal L$-definable, then so is $B\cap V\cap X$.
\end{proof}

\begin{lemma}\label{Ldef}
A locally $\cal L$-definable set is $\cal L$-definable.
\end{lemma}
\begin{proof}
Let $V\sub M^n$ be locally $\cal L$-definable and suppose that the closure $C=cl(V)$ has dimension  $k$. We work by induction on $k$. If  $k=0$, then $V$ is finite and hence $\cal L$-definable. Suppose $k>0$. We first prove that
$$\dim cl(C\sm V)<k.$$
If not, there is a $k$-cell $C'\sub C$ in which $C\sm V$ is dense. Since $V$ is dense in $C$, by Fact \ref{denseZ0} it is also dense in $C'$. In particular, there is $x\in C'\cap V$. Since $C'\cap V$ is locally $\cal L$-definable (Lemma \ref{VcapC}), there is a $k$-cell $B\sub C'$ containing $x$ such that $B\cap V$ is $\cal L$-definable. But then both $B\cap V$ and $B\cap (C\sm V)$ are dense in $B$. That is, $B\cap V$ and $B$ are both $\cal L$-definable and the former is dense and co-dense in the latter. A contradiction.

By Lemma \ref{VcapC}, the set $cl(C\sm V)\cap V$ is locally $\cal L$-definable. Since its closure is contained in $cl(C\sm V)$,  by inductive hypothesis we obtain that it is $\cal L$-definable. Since
$$V\sm cl(C\sm V)= C\sm cl(C\sm V)$$
is also $\cal L$-definable, we conclude that $V$ is $\cal L$-definable.
\end{proof}


Although it will not be used in this paper, we note that local $\cal L$-definability at a point is a definable notion.

\begin{lemma}\label{Sdefinable}
Let $V\sub M^n$ be a definable set. Then the set $S$ of points in $V$ at which $V$ is locally $\cal L$-definable is $\cal L$-definable. 
\end{lemma}
\begin{proof} It is easy to see that for any point $x\in V$, we have that $V$ is locally $\cal L$-definable at $x$ if and only if there is a closed box $B\sub M^n$ containing $x$ such that $cl(B\cap V)=B\cap V$. 
Therefore,  $S$ is definable. By its definition, it is thus locally $\cal L$-definable. By Lemma \ref{Ldef}, it is $\cal L$-definable.
\end{proof}

 \subsection{Extracting local $\cal L$-definability}

A simple and illustrative example 
of what follows is this. Let $\Gamma$ be the set of non-algebraic real numbers, and $f:\R^2\to \R$  the usual addition. Then $f(J\times J)=\R$ is $\cal L$-definable. The statements that follow generalize this observation. The extra complication in proving Corollary \ref{coordinate inj3} below is due to the fact that the domain $U$ of the given $f$ is not a product of $\cal L$-definable sets, forcing us to first prove local $\cal L$-definability of $f(\Gamma^2 \cap U)$, with the assistance of the preceding lemmas.


\begin{lemma}\label{coordinate inj}
Let $S_1, S_2\sub M^n$ be two $k$-cells, and $\Gamma_i\sub S_i$ definable sets with $\ldim(S_i\sm \Gamma_i)<k$, for $i=1,2$. Suppose $f:S_1\times S_2\to M^k$ is an $\cal L$-definable continuous map,   injective in each coordinate, and let $x\in f(\Gamma_1\times \Gamma_2)$. Then $f(\Gamma_1\times \Gamma_2)$ contains an open subset of $M^k$ that contains $x$.
\end{lemma}
\begin{proof}
The proof is inspired by an example in \cite[page 5]{beg}. Let $x=f(a, b)$, where $(a, b)\in\Gamma_1\times \Gamma_2$. We first claim that there is a definable set $\Gamma'_1\sub \Gamma_1$ of  dimension $k$, such that $\bigcap_{t\in \Gamma'_1} f(t, S_2)$
contains an open set $K$ that contains $x$. Since $f(a, -):S_2\to M^k$ is $\cal L$-definable, continuous and injective, by \cite{johns} $f(a, S_2)$ contains an open set that contains $f(a, b)$. By continuity of $f$, there is a $k$-cell $I\sub S_1$ containing $a$, such that $\bigcap_{t\in I} f(t, S_2)$ contains an open set that contains $f(a, b)$. We can thus set $\Gamma'_1=I\cap \Gamma_1$, which has dimension $k$.

Now let $\Gamma'_1$ and $K$ be as above. We prove that actually $f(\Gamma'_1\times \Gamma_2)$ contains $K$. Assume towards a contradiction that there is $m_0\in K$ such that $m_0\not\in f(\Gamma'_1\times \Gamma_2)$. By the claim in the first paragraph,
$$m_0\in \bigcap_{t\in \Gamma'_1} f(t, S_2\sm \Gamma_2).$$
By injectivity of $f$ in the first coordinate, we obtain
$$\Gamma'_1\sub f(-, S_2\sm \Gamma_2)^{-1}(m_0)$$
which is a contradiction, because $\ldim(S_2\sm \Gamma_2)<k=\ldim (\Gamma'_1)$, and hence its image under the definable map $x\mapsto f(-, x)^{-1}(m_0)$ cannot contain $\Gamma'_1$ (see, for example, \cite[Corollary 5.3]{egh}). 
\end{proof}

We next derive a version of the last lemma where the range of $f$ is a $k$-cell in any $M^n$.


\begin{lemma}\label{coordinate inj2} Let $S_1, S_2, V\sub M^n$ be three $k$-cells, and $\Gamma_i\sub S_i$ definable sets with $\ldim (S_i\sm \Gamma_i)<k$, for $i=1, 2$. Suppose $f:S_1\times S_2\to V$ is an $\cal L$-definable continuous map, injective in each coordinate, and let $x\in f(\Gamma_1\times \Gamma_2)$. Then $f(\Gamma_1\times \Gamma_2)$ contains a $k$-cell that contains $x$.
\end{lemma} 
\begin{proof} Suppose $x=f(a, b)$, for $(a, b)\in \Gamma_1\times \Gamma_2$.
Let $\pi:M^n\to M^k$ be  a coordinate projection which is injective on $V$. Then $F:=\pi\circ f:S_1\times S_2\to M^k$ is an $\cal L$-definable continuous map that satisfies the conditions of Lemma \ref{coordinate inj}. So $F(\Gamma_1\times \Gamma_2)$ contains an open box $K$ of $M^k$ that contains $\pi(x)$. Then $(\pi_{\res V})^{-1}(K)\sub f(\Gamma_1\times \Gamma_2)$ is a $k$-cell that contains $x$.\end{proof}


It is not hard to see that the above lemma remains true if $V$ is any $\cal L$-definable set of dimension $k$, but we will not need this fact here. However, if $V$ is of higher dimension,  then the lemma fails: let $f$ be the identity map and $\Gamma$ contain no open $\cal L$-definable set.\\

We can now prove the exact statement that will be used  in the proof of Lemma \ref{F(V)}.

\begin{cor}\label{coordinate inj3}  Let $V\sub M^n$ be a $k$-cell, $U\sub V^2$ a finite union of $2k$-cells, and $\Gamma\sub V$ a definable set with $\dim (V\sm \Gamma)<k$. Suppose that $f:U\to V$ is an $\cal L$-definable continuous map, which is injective in each coordinate. Then the set $X=f(\Gamma^2\cap U)$ is $\cal L$-definable.
\end{cor}
\begin{proof}

By Lemma \ref{Ldef}, it suffices to show that   $X$ is locally $\cal L$-definable. So let $x=F(a, b)$, where $(a, b)\in \Gamma^2\cap U$. Since $U\sub V^2$ is a finite union of $2k$-cells and $\dim V=k$, it is easy to find $k$-cells $S_1, S_2\sub V\sub M^n$  such that $(a, b)\in S_1\times S_2\sub U$. Since $S_i\sub V$, we have $\ldim (S_i\sm\Gamma)<k,$ and hence by Lemma \ref{coordinate inj2}, for $\Gamma_i=\Gamma\cap S_i$, there is a $k$-cell $B$ with
$$x\in B\sub F((\Gamma\cap S_1)\times (\Gamma\cap S_2))\sub X.$$
By  Fact \ref{locLdeffact}, $X$ is locally $\cal L$-definable.
\end{proof}

\section{A Weil's group chunk theorem}\label{sec-chunk}

The goal of this section is to recover an $\cal L$-definable group from an $\cal L$-definable `group chunk'. Theorems of this spirit have already been considered in classical model theory. The current account borrows ideas from Weil's group chunk theorem as it appears in van den Dries \cite{vdd-weil}. The proof of Theorem \ref{chunk} is based on discussions with Y. Peterzil.

In this section, we work in a more general setting than in the rest of this paper. Let $\cal M$ and $\cal N$ be any two first-order structures, with $\cal N$ expanding \cal M. Assume that there is a map $\dim$ from the class of all definable sets in \cal N to $\{-\infty\}\cup \N$, such that the following properties from the introduction hold:
\begin{itemize}
\item \textbf{(D1), (D2), (D3b), (D4)}, and
\item if the family  $\{X_t\}_{t\in I}$  in \textbf{(D3)} is \cal L-definable, then so are the sets $I_d$ in \textbf{(D3a)}.
\end{itemize}
We refer to the second property above as \textbf{(Ldef)}. Note that we do not assume that \cal M is o-minimal, nor that it admits a dimension function. But even for an o-minimal $\cal M$, the current setting is much richer than in the rest of the paper. For example, it includes $d$-minimal structures (\cite{for}), such as $\cal N=\la \R, <, +, \cdot, 2^\Z\ra$, and also weakly o-minimal non-valuational structures (\cite{wen}), such as $\cal N=\la \Q, <, +, (0, \pi)\ra$.

\begin{remark}\label{rmk-sbd} An example where  \cal N satisfies the above properties, but not \textbf{(D3a)}, is that of a \emph{semi-bounded} o-minimal structure $\cal N=\la \cal M, P\ra$; namely, when \cal M is a linear o-minimal structure, $P$ is an o-minimal expansion of a real closed field defined on a bounded interval, and $\dim$ is the usual o-minimal dimension (\cite[Proposition 3.6]{pet-sbd}). Even though groups definable in semi-bounded o-minimal structures are already well-understood (\cite{ep-sel2}), the results of this section appear to be new also  in that setting.
\end{remark}

Under these assumptions, we  recover a group which is \emph{interpretable} in $\cal M$. This will be enough for our purposes in this paper, in view of Fact \ref{fact-epr} below. To avoid any ambiguities, let us recall the following definition from \cite{epr}.

\begin{defn}\label{def-int} Let $\cal R$ be any structure. By a \emph{definable quotient} (in \cal R) we mean a quotient $X/E$ of a definable set $X$ by a definable equivalence relation $E$.
A map $f:X/E_1\to Y/E_2$ between two definable quotients is called \emph{definable} if the set
$$\{(x,y)\in X\times Y: f([x])=[y]\}$$ is definable (in \cal R). An \emph{interpretable group} $G$ is a group whose universe is a definable quotient, and whose group operation is a definable map.
\end{defn}

\begin{fact}[{\cite[Theorem 1]{epr}}]\label{fact-epr}
If $\cal R$ is o-minimal, then every interpretable group is definably isomorphic to a definable group.
\end{fact}



We extend our terminology from the introduction to definable quotients:   a quotient, map and group as in Definition \ref{def-int}, is called `definable' or `interpretable' if $\cal R=\cal N$, and `$\cal L$-definable' or `\cal L-interpretable' if $\cal R=\cal M$.


\begin{thm}\label{chunk}
Let $G=\la G, \cdot, 1\ra$ be a definable group with $G\sub M^n$ and $\dim G=k$. Suppose that
\begin{itemize}
\item $X\sub M^m$ and $Z\sub X^2$ are two $\cal L$-definable sets,  with $\dim (X^2\sm Z)<2k$,
\item $h:X\to G$ is a definable injective map, with $\dim(G\sm h(X))<k$, and
\item $F:X^2\to X$ is an $\cal L$-definable map, such that for every $(x,y)\in Z$,
$$h(x) \cdot h(y)=h(F(x,y)).$$
\end{itemize}
Then $G$ is definably isomorphic to an $\cal L$-interpretable group.
If, moreover, \cal M is o-minimal, then $G$ is definably isomorphic to an $\cal L$-definable group.

\end{thm}

\begin{proof} We may assume that $X\sub G$ and $h=id$. Indeed, one can form the disjoint union  of $X$ and $G\sm h(X)$, and induce on it a definable group structure  after identifying $X$ with $h(X)$, and $G\sm h(X)$ with itself. We then have $\dim(G\sm X)<k$, $\dim (X^2\sm Z)<2k$, and   $F: X^2\to X$ is an $\cal L$-definable map,   such that for every $(x,y)\in Z$,
$$x\cdot y=F(x,y).$$
To simplify the notation, for $a,b\in G$, we  may write $ab$ for $a\cdot b$. Moreover, we may assume that for every $a\in X$, $\dim(X\sm Z_a)<k$.
Indeed, by \textbf{(Ldef)}, the set $X'=\{a\in X : \dim (X\sm Z_a)<k\}$ is \cal L-definable. Moreover, since $\dim(X^2\sm Z)<2k$, it follows that $\dim (X\sm X')<k$. We may thus replace $X$ by $X'$. Finally, note  that $F_{\res Z}$ is injective in each coordinate.\\

\noindent\textbf{Claim 1.} {\em The following sets are $\cal L$-definable:}\smallskip

$A=\{(a,b,c,d)\in X^4: ab=cd\}$,

  $B=\{(a, b)\in X^2 : ab=1\}$.

\begin{proof}[Proof of Claim 1]
For $A$, it suffices by \textbf{(Ldef)} to prove that for every $a,b,c,d\in X$, $ab=cd$ if and only if the set
$$\{x\in X: F(a, F(b, x))= F(c, F(d, x))\}$$
has co-dimension $<k$ in $X$. To see this, it suffices to show that for every $a, b\in X$, the set
$$\{x\in X : a b x=F(a, F(b,x))\}$$
has co-dimension $<k$ in $X$. Clearly, every $x\in Z_b$ such that $F(b,x)\in Z_a$ is contained in the last set. But by injectivity of $F_{\res Z}$ in the second coordinate, there are co-dimension $<k$ many such $x$.

 For  $B$, one can see similarly that for every $a,b\in X$, $(a,b)\in B$ if and only if the set
$$\{x\in X : F(a, F(b, x))= x\}$$
has co-dimension $<k$ in $X$.
\end{proof}

Let $\sim$ be the following equivalence relation on $X^2$:
$$(a, b) \sim (c,d) \,\,\,\Lrarr\,\,\, ab=cd.$$
By  definability of the set $A$ from Claim 1, the relation $\sim$ is $\cal L$-definable. Let $K=X^2/\sim$ and denote by $[(a,b)]$ the equivalence class of $(a,b)$. So $K$ is an $\cal L$-definable quotient.
We  aim to equip $K$ with an $\cal L$-interpretable group structure $\la K, *, 1_K\ra$.\\ 

\noindent\textbf{Claim 2.}
(1)  {\em For every $a,b,c,d\in X$, there are $e, x, y, f\in X$, such that $ab=ex$, $cd=yf$ and $xy=1$.}

 (2)  {\em For every $a,b, c,d,s,t\in X$, there are $e,x,y,z,w,f\in X$, such that $ab=ex$, $cd=yz$, $st=wf$ and $xy=zw=1$.}
\begin{proof}[Proof of Claim 2] We only prove (1), as the proof for (2) is similar.
Consider the sets
$$S=\{(e,x)\in X^2 : ab=ex\}$$
and
$$T=\{(y,f)\in X^2: cd=yf\}.$$
Since $\dim(G\sm X)<k$, the projections $\pi_1(T)$ and $\pi_2(S)$ on the first and last $m$ coordinates, respectively, have co-dimension $<k$ in $X$. In particular,
$$\pi_2(S)\cap X\cap (\pi_1(T)\cap X)^{-1}\ne \es.$$
Now take $x$ in this set and let $y=x^{-1}$, $e=abx^{-1}$ and $f=xcd$. By construction, $x, y, e, f\in X$ and they satisfy the equalities of the conclusion.
\end{proof}

Now, for every $a,b,c,d, e,f\in X$, define the relation
$$R(a,b,c,d,e,f)\,\,\Lrarr \, \text{there are $x,y\in X$ such that $ab=ex$, $cd=yf$ and $xy=1$.}$$
By Claim 1,  $R$ is an $\cal L$-definable relation. By Claim 2(1), for every $a,b,c,d\in X$, there are $e,f\in X$ such that $R(a,b,c,d,e,f)$. Moreover, if $R(a,b,c,d,e,f)$, $R(a',b',c',d',e',f')$, $ab=a'b'$ and $cd=c'd'$, then $ef=e'f'$. Indeed, let $x,y, x',y'$ witnessing the first two relations. Then
$$ef=exyf=abcd=a'b'c'd'=e'x'y'f'=e'f'.$$
We can thus define the following $\cal L$-definable operation on $K$:
$$[(a,b)]*[(c,d)]=[(e,f)]\,\,\Lrarr\,\, R(a,b,x,d,e,f).$$
Let $1_K=[(x,y)]$ for some/any $x,y\in X$ such that $xy=1$. Namely, take $x\in X\cap X^{-1}$, which exists since $\dim (G\sm X)<k$.

$ $\\
\noindent\textbf{Claim 3.} {\em $K$ is an $\cal L$-definable group.} 
\begin{proof}[Proof of Claim 3]
We already saw that $K$ and $*$ are $\cal L$-definable. We prove associativity of $*$. Let $a,b,c,d,e,f\in X$. Take  $e,x,y,z,w,f$ as in Claim 2(2). Then
\begin{align}
  ([(a,b)]*[(c,d)]) *[(s,t)]=[(e,z)]*[(w,f)]=[(e,f)] &=[(e,x)]*[(y,f)]   \notag\\
  &= [(a,b)]*([(c,d)] *[(s,t)]).\notag
\end{align}
It is also easy to check that $1_K$ is the identity element, using Claim 2(1).
 \end{proof}

\noindent\textbf{Claim 4.} {\em $K$ is definably isomorphic to $G$.} 
\begin{proof}[Proof of Claim 3]
Let $f:K\to G$ be given by $[(a,b)]\mapsto ab$. By definition of $\sim$, $f$ is injective. It is also onto since $\dim(G\sm X)<k$ and hence for every $x\in G$, we can choose $a\in X\cap x X^{-1}$ and $b=a^{-1}x$. It remains to see that $f$ is a group homomorphism. Let $a,b,c,d\in X$, and take $e,x,y,f\in X$ as in Claim 2(1). We then have
$$f( [(a,b)]*[(c,d)])= f([(e,f)])=ef= exyf=abcd = f([(a,b)]) \cdot  f([(c,d)]),$$
as required.
\end{proof}
The `moreover' clause is clear by the definitions and Fact \ref{fact-epr}.
\end{proof}

\section{The proof of Theorem \ref{main}.}\label{sec-proof}

We are now ready to prove Theorem \ref{main}.   The left-to-right direction is immediate, since every $\cal L$-definable group is strongly large. For the right-to-left direction, we prove that any strongly large group satisfies the assumptions of Theorem \ref{chunk}.

\begin{thm}\label{main_chunk} Let $G=\la G, \cdot, 1\ra$ be a strongly large group with $G\sub M^n$ and $\dim G=k$. Then there are
\begin{itemize}
\item \cal L-definable sets $X\sub M^m$ and $Z\sub X^2$,  with $\dim (X^2\sm Z)<2k$,
\item  a definable injective map $h:X\to G$, with $\dim(G\sm h(X))<k$, and
\item   an $\cal L$-definable map $F:X^2\to X$, such that for every $(x,y)\in Z$,
$$h(x) \cdot h(y)=h(F(x,y)).$$
\end{itemize}
\end{thm}

The rest of this section is devoted to proving Theorem \ref{main_chunk}. The proof runs through the five first steps mentioned in the introduction. For $a, b\in G$, we write $a b$ for $a\cdot b$.

 \bigskip
\noindent\textbf{Step I : Recovering an $\cal L$-definable map $F$ from $\cdot$.}\\


The results on strongly large sets from Section \ref{sec-sl} are used here. Note that in the next lemma, if the set $C$ were $\cal L$-definable,   we would  already have proved Theorem \ref{main_chunk} (with $X=V$, $Z=C$, and $h=id$).

\begin{lemma}\label{firststep}
There are a closed $\cal L$-definable  set $V\sub M^n$, with $\dim (G\triangle V)<k$, a definable set $C\sub (V\cap G)^2$, with $\dim(V^2\sm C)<2k$, and an $\cal L$-definable map $F:V^2\to M^n$, such that
$$F_{\res C}= \cdot_{\res C}.$$
Moreover,
$$C=\bigcup_{t\in \Gamma} \{t\}\times C_t,$$
for some $\Gamma, C_t\sub V$, for $t\in \Gamma$, with $\dim(V\sm \Gamma)<k$ and  $\dim (V\sm C_t)<k$.
\end{lemma}

\begin{proof}By Corollary \ref{slfull},  $G$ is the union of a full set $G_1$ and a set of dimension $<k$.  
Let $V=cl(G_1)$. Since $G_1$ is full,
 $\dim (V\sm G_1)<k$ and hence $\dim (V\sm G)<k$. Since also $\dim (G\sm V)\le \dim(G\sm G_1)<k$, we obtain $\dim (G\triangle V)<k$. Now, by Lemma \ref{slF}, there is a definable set $S\sub G_1^2$ of  dimension $<2k$ and an $\cal L$-definable map $F:M^2\to M^n$ that agrees with $\cdot$ on $G_1^2\sm S$.
 Let
$$\Gamma=\{t\in G_1: \ldim (S_t)<k\},$$
and for $t\in \Gamma$,
$$C_t=G_1\sm S_t.$$
Let also $C=\bigcup_{t\in \Gamma} \{t\}\times C_t$.
By \textbf{(D3)}, $\ldim (G_1\sm \Gamma) <k$, and  since $\dim (V\sm G_1)<k$, $\dim (V\sm \Gamma)<k$. Similarly, $\dim (V\sm C_t)<k$, for $t\in \Gamma$.
It follows from \textbf{(D3)} that $\dim(V^2\sm C)<2k$.
 Finally, since $C\sub G_1^2\sm S$, we have $F_{\res C}= \cdot_{\res C}.$
\end{proof}

\emph{For the rest of this section, we fix $ V, \Gamma, C_t, C$ and $F$ as above, and  use their properties without any specific mentioning.}\\

\noindent Note that since $\dim (G\sm V)<k$, it follows  that $\dim (G^2\sm C)<2k$, $\dim (G\sm \Gamma)<k$, and, for $t\in \Gamma$, $\dim (G\sm C_t)<k$.\\

\noindent\textbf{Step II: Group-like properties of $F$.}\\

Here we prove the existence of an  $\cal L$-definable set $U\sub V^2$, with $\dim(V^2\sm U)<2k$, on which $F$ is continuous and behaves like a group operation, with $F(U)\sub V$. This is done through a series of lemmas.




\begin{lemma}\label{Wcont}
There is an $\cal L$-definable set $U\sub V^2$, with $\dim(V^2\sm U)<2k$, such that $F_{\res U}$ is continuous and $F(U)\sub V$.
\end{lemma}
\begin{proof}
By o-minimality, there is an $\cal L$-definable set  $U\sub V^2$, which is a finite union of $2k$-cells, with $\dim(V^2\sm U)<2k$, such that $F_{\res U}$ is continuous. We claim that $F(C')\sub V$, for some set $C'\sub V$ which is dense in $U$. Indeed, for every $t\in \Gamma$, consider the set
$$C'_t=\{x\in C_t : F(t,x)\in  \Gamma\}. $$
Since $F_{\res C} = \cdot_{\res C}$, we have that $F(t, -)_{\res C_t}$ is injective. Since also $\ldim (G\sm \Gamma)<k$, it follows from \textbf{(D4)} that $\dim(C_t\sm C'_{t})<k$. Hence $\ldim (V\sm C'_t)<k$. Moreover, $F(t, C'_t)\sub \Gamma\sub V$.
Let $$C'=\bigcup_{t\in \Gamma}\{t\}\times C'_t.$$ By \textbf{(D3)},  $\ldim (V^2\sm C')<2k$, and hence $\ldim(U\sm C')<2k$. By Fact \ref{denseldim}, since  $U$ is a finite union of $2k$-cells, $C'$ is dense in $U$. Moreover, $F(C')\sub V$, as required. Now, since $U\sub cl(C')$, $V$ is closed and $F_{\res U}$ is continuous,  it follows that $F(U)\sub V$.
\end{proof}



\begin{lemma}\label{Finj}
There is an $\cal L$-definable set $U\sub V^2$ with $\dim (V^2\sm U)<2k$, such that $F_{\res U}$ is injective in each coordinate.
\end{lemma}
\begin{proof}
It suffices to find an $\cal L$-definable set $U\sub V^2$ with $\dim (V^2\sm U)<2k$, such that $F_{\res U}$ is injective in the second coordinate. One can then similarly find $U'\sub V^2$ with $\dim (V^2\sm U')<2k$ and $F_{\res U'}$  injective in the first coordinate, and the intersection $U\cap U'$ is the desired set.

Suppose towards a contradiction that there is no such $U$. For every $t\in V$, let
$$N(t)=\{x\in V: \exists y\in V,\, y\ne x \text{ and } F(t, x)=F(t, y)\}.$$ Then the set $$K=\{t\in V:  \dim N(t) =k\}$$ is $\cal L$-definable. So for any  $t\in K$, $F(t, -)$ is not injective on any $\cal L$-definable subset of $V$ of co-dimension $<k$. So, by assumption and \textbf{(D3)}, $\dim K= k$. Since $\Gamma $ is dense in $V$,  there is $t\in K\cap \Gamma\ne \es$. Since $t\in \Gamma$, $F(t, -)_{\res C_t}$ is injective. By Lemma \ref{inj-lemma}, for $D=C_t$ and $X=V$, there is an $\cal L$-definable subset of $V$ of co-dimension $<k$ on which $F(t, -)$ is  injective, contradicting $t\in K$.
\end{proof}

\begin{lemma}\label{W}
There is an $\cal L$-definable set $W\sub V^3$ with $\dim (V^3\sm W)< 3k$, such that for every $(x, y, z)\in W$,
$$(*)\,\,\,\,\, F(x, F(y, z))= F(F(x, y), z).$$
\end{lemma}
\begin{proof} By o-minimality, there is an $\cal L$-definable set  $W\sub V^3$, which is a finite union of $3k$-cells, with $\dim (V^3\sm W)<3k$, such that both maps $F(-, F(-, -)):V^3\to M^n$ and $F(F(-, -), -):V^3$ are  continuous on $W$.
It is thus enough to prove that (*) holds on a dense subset of $W$. 

We observe that for every $(x, y, z)\in V^3$ with $x, y, F(x,y)\in \Gamma$, $F(y,z)\in C_x$ and $z\in C_y\cap C_{F(x,y)}$,  equation (*) holds, since both of its sides equal $xyz$. Hence, if, for every $x, y\in V$, we let
$$Y_x=\Gamma\cap F(x, -)^{-1}(\Gamma)$$
and
$$Z_{x, y}= F(y, -)^{-1}(C_x)\cap C_y \cap C_{F(x,y)},$$
then (*) holds on the set
$$T=\bigcup_{(x,y)\in \bigcup_{x\in \Gamma} \{x\}\times Y_x} \{(x,y)\}\times Z_{x,y}.$$

Since $F_{\res C}=\cdot_{\res C}$, $\dim (G\sm \Gamma)<k$ and, for $x\in \Gamma$, $\dim (G\sm C_x)<k$, it follows easily from \textbf{(D4)} that for every $x\in \Gamma$ and $y\in Y_x$, $\ldim (G\sm Y_x)<k$ and $\ldim(G\sm Z_{xy})<k$.
Since also $\ldim(G\sm \Gamma)<k$,  \textbf{(D3)} implies that $\ldim (G^3\sm T)<3k$, and hence $\dim (V^3\sm T)<3k$. Thus $\ldim (W\sm T)<3k$, and by Fact \ref{denseldim}, $T$ is dense in $W$.
\end{proof}

We can refine the set $U$ in order to achieve two additional properties.

\begin{cor}\label{collectF} Let $W$ be as in Lemma \ref{W}. Then, there is an $\cal L$-definable set $U\sub V^2$, such that
\begin{itemize}
  \item $\dim (V^2\sm U)<2k$, 
  \item $F_{\res U}$ is continuous, and injective in each coordinate,
  \item $F(U)\sub V$,
  \item for every $t\in V$, $\dim(V\sm U_t)<k$,
  \item for every $(t, x)\in U$, $\ldim(V\sm W_{t,x})<k$. 
\end{itemize}
\end{cor}
\begin{proof} Let $U$  be as in Lemma \ref{Wcont}. Define $$T=\{t\in V : \dim (V\sm U_t)<k\}$$ and
$$U_1=\bigcup_{t\in T} \{t\}\times U_t.$$
Since $\dim (V^2\sm U)<2k$, we obtain $\dim (V\sm T)<k$ and $\dim (V^2\sm U_1)<2k$. 
Define also
$$U_2=\{ (t, x)\in V^2: \dim (V\sm W_{t,x})<k\}.$$
Since $\dim (V^3\sm W)<3k$, we obtain $\dim (V^2\sm U_2)<2k$. 
The desired set $U$ is the intersection of $U_1$, $U_2$ and the set obtained in Lemma \ref{Finj}.
\end{proof}

\emph{For the rest of this section, we fix the sets $U$ and $W$ as above, and  use their properties without any specific mentioning.}\\


\noindent\textbf{Step III: Extracting an $\cal L$-definable set $X\sub V$ using $F$.}\\


In this step, we  use $F$ to recover a suitable $\cal L$-definable set $X\sub V$ with $\dim (V\sm X)<k$. The work from Section \ref{sec-localLdef} plays an essential role here. The suitability of $X$ will be evident in Step IV.

\begin{lemma}\label{F(V)}
 The set $X=F(\Gamma^2 \cap U)$ is  $\cal L$-definable with $\dim(V\sm X)<k$.
 \end{lemma}
\begin{proof}
By Corollary \ref{coordinate inj3}, $X$ is \cal L-definable, so we need to show that $\dim(V\sm X)<k$. By cell decomposition, $V$ is a finite union of cells. Let $V'$ be the union of all $2k$-cells in this decomposition. We write $V$ for $V'$. Since $X$ is $\cal L$-definable, it suffices to show that $X$ is dense in $V$. Pick any $t\in \Gamma$ and write $K=U_t\cap C_t$. So $\dim (\Gamma \sm K)<k$. On the one hand, we have
$$\Gamma\sm t K\sub t (G\sm K)\sub t (\Gamma\sm K)\cup t (G\sm \Gamma)$$
and hence $\ldim (\Gamma\sm t K)<k$. Since $\ldim (V\sm \Gamma)<k$, we obtain $\ldim (V\sm t K)<k$. By Fact \ref{denseldim},  $t K$ is dense in $V$. On the other hand,
$$F(t, K)=t K.$$
That is,  $F(t, U_t\cap C_t)$ is dense  in $V$, and hence so is $X$.
\end{proof}




\smallskip
\noindent\textbf{Step IV: Constructing a definable embedding $h:X \to G$.}\\

In this step, we embed $X$ into $G$, after proving the key property (*) from the introduction.
For every $t\in X$ and $r\in G$, the set
$$L_{t,r}=\{x\in V\cap G: F(t,x)=r x\}$$
is definable.



\begin{lemma}
For every $s\in X$, there is unique $r\in G$, such that $\ldim(G\sm L_{s,r})<k$.
\end{lemma}
\begin{proof}
Let $s=F(t, x)$, where $(t,x)\in \Gamma^2\cap U$. Recall that $\ldim(V\sm W_{t,x})<k$. Let
$$Y=W_{t,x}\cap F(x, -)^{-1}(C_t)\cap C_x.$$
Then $\ldim (G\sm Y)< k$. Moreover, for every $y\in Y$, we have
$$F(F(t,x), y)=F(t, F(x,y))=t F(x, y)=txy.$$
That is, for $r=tx$, we obtain $Y\sub L_{s,r}$. It follows that $\ldim (G\sm L_{s,r})<k$.

The uniqueness of $r$ is clear, since otherwise we would obtain two sets $L_{s, r}$ and $L_{s, r'}$ both contained in $G$ and having co-dimension $<k$ in $G$, a contradiction.
\end{proof}

We now consider the map  $h:X\to G$ given by
$$h(t) = r \,\Lrarr\, \ldim(G\sm L_{t,r})<k \,\,\, (\text{equivalently, }  \ldim (V\sm L_{t,r})<k).$$

Recall that  for every $t\in X\sub V$, $\dim(V\sm U_t)<k$.

\begin{claim}\label{ginj}
$h$ is injective.
\end{claim}
\begin{proof}
Suppose that for $t, s\in X$, we have $h(t)=h(s)=r$. Then $\ldim(V\sm L_{t,r})<k$ and $\ldim(V\sm L_{t,s})<k$. Therefore,
$$U_t\cap U_s\cap L_{t, r}\cap L_{s, r}\ne \es.$$
For $x$ in that intersection, we have $F(t,x)=rx=F(s,x)$, and by injectivity of $F_{\res U}$ in the first coordinate, $t=s$.
\end{proof}


\begin{claim}\label{gGamma}
  For every $t\in \Gamma$, $\ldim(G\sm L_{t,t})<k$. In particular, $h_{\res X\cap \Gamma}=id$.
\end{claim}
\begin{proof}
Let $t\in \Gamma$. Then for every $x\in C_t$, $F(t,x)=t x$. 
So, $C_t\sub L_{t,t}$. Since $\ldim (G\sm C_t) <k$, the result follows.
\end{proof}

Since $\dim (V\sm X)<k$, we have $\dim (G\sm X)<k$. Since also $\dim (G\sm \Gamma)<k$, we obtain $\dim(G\sm X\cap \Gamma)<k$. Therefore, by Claim \ref{gGamma}, $\dim (G\sm h(X))<k$.

$ $\\
\noindent\textbf{Step V: Concluding the proof of Theorem \ref{main_chunk}.}\\



It remains to show the following statement.

\begin{lemma}\label{Fhom} There is an $\cal L$-definable set $Z\sub X^2$ with $\dim(X^2\sm Z)<2k$ and $F(Z)\sub X$,  such that for every $(t,x)\in Z$,
$$h(F(t,x))=h(t)h(x).$$
\end{lemma}
\begin{proof} We let
$$Z=F^{-1}(X)\cap X^2\cap U.$$
Clearly, $F(Z)\sub X$. We prove $\dim (X^2\sm Z)<2k$. Recall that $X\sub V$, 
and hence $\dim (X^2\sm U)<2k$. So, it suffices to prove  that $\dim (X^2\sm F^{-1}(X))<2k$.  Let $t\in X\sub V$. Since $\dim(V\sm (\Gamma\cap U_t))<k$, we have  $\dim(X\sm (\Gamma \cap U_t))<k$. Let $\Gamma'=\Gamma \cap X$. Then for every $t\in \Gamma'$, $\dim (X\sm (\Gamma\cap U_t))<k$, and
$$F(t, \Gamma \cap U_t)\sub F(\Gamma^2 \cap U)=X.$$
Hence, the set
$$\bigcup_{t\in \Gamma'} \{t\} \times (\Gamma \cap U_t)$$
belongs to $F^{-1}(X)$ and has co-dimension $<2k$ in $X^2$, as required.

Now let $(t, x)\in Z$. For  $y\in X$, denote $D_y=L_{y, h(y)}$. So $\ldim(V\sm D_y)<k$. By injectivity of $F_{\res U}$ in the second coordinate, $F(x, -)^{-1}(D_t)$ has co-dimension $<k$ in $V$.
Hence the set
$$Y= D_{F(t,x)}\cap W_{t,x}\cap  F(x, -)^{-1}(D_t) \cap D_x$$ 
is non-empty. Take any $y\in Y\sm \{1\}$. Then
$$h(F(t,x)) y= F(F(t,x), y)=F(t,F(x,y))=h(t) F(x,y)=h(t)h(x)y,$$
and hence
$h(F(t,x))=h(t)h(x)$, as required.
\end{proof}


This ends the proof of Theorem \ref{main_chunk}, and hence, by Theorem \ref{chunk}, also that of Theorem \ref{main}.



\section{Expansions by dense independent sets}\label{sec-indep}


In this section, we let  $\cal M=\la M, <, +, \dots\ra$ be an o-minimal expansion of an ordered group, $P\sub M$  a dense $\dcl$-independent set, and $\cal N=\la \cal M, P\ra$. We let $\dim$ be the large dimension coming from \cite{egh}, as described in Section \ref{sec-egh}. Note that the assumption that $\cal M$ expands a group is not due to any reasons pertaining the current work, but only because the accounts \cite{dms2} and \cite{egh} that analyze this pair work under it.

Theorem \ref{main2} will follow from Theorem \ref{main} and the following theorem.

\begin{theorem}\label{indsl}
Every definable group is definably isomorphic to a strongly large group.
\end{theorem}


The rest of this section is devoted to proving Theorem \ref{indsl}.  The proof  is based on the cone decomposition theorem from \cite{egh}. The terminology of Section \ref{sec-egh} applies here. A simplified formulation of the cone decomposition theorem is that every definable set $X\sub M^n$ is a small union of sets of the form $h(J)$, where $J$ is a supercone in some $M^k$, and $h:J \to M^n$  is an \cal L-definable continuous injective map. However, one can achieve some uniformity in the above decomposition, by stocking the different $J$'s into finitely many families of supercones, each in a fixed $M^k$, and extending $h$ to every such family $\cal L$-definably and continuously.
For $\cal J=\bigcup_{g\in S}\{g\}\times J_g$, we write $S=\pi(\cal J)$.

\begin{defn}[Cones]\label{def-cone}
A set $C\sub M^n$ is a \emph{$k$-cone}, $k\ge 0$,  if there is a definable set   $\cal J=\bigcup_{g\in S}\{g\}\times J_g$, where $S\sub P^m$ and every $J_g\sub M^k$ is a supercone, and an $\cal L$-definable continuous map $h:V \subseteq M^{m+k}\to M^n$, where $V$ is cell, such that
\begin{enumerate}
\item for every $g\in S$, $V_g=sh(J_g)$,
\item $C=h(\cal J)$,
\item $h:\cal J\to M^n$ is injective.
\end{enumerate}
A \emph{cone} is a $k$-cone for some $k$.
\end{defn}


\begin{remark}\label{hJfull}
It is important to note that if $C=h(\cal J)$ is a cone as above, then for every $g\in S$, the set $h(g, J_g)$ is a full set (as in Section \ref{sec-sl}). Indeed, the map $h(g, -) : sh(J_g)\to M^n$ is $\cal L$-definable continuous and injective. By Fact \ref{exa-hJ}, $h(g, J_g)$ is a full set.
\end{remark}

\begin{fact}[Cone decomposition theorem]\label{conedec} Every definable set is a finite disjoint union of cones.
\end{fact}
\begin{proof}
This is a consequence of the cone decomposition theorem in \cite{egh} and subsequent work in \cite{egh2}. A detailed proof is given in \cite[Fact 4.7]{el-pw}. In that reference the universe of $\cal M$ is assumed to be $\R$, but this played no role in the particular proof.
\end{proof}

We will need a further decomposition as follows.

\begin{claim}\label{partitionG} Every $k$-cone is a finite disjoint union of $k$-cones $h(\cal J)$, with $S=\pi(\cal J)$, such that: 
\begin{enumerate}
     \item   every $g\in S$ has all its coordinates distinct,
      \item $S$ is either finite, or every coordinate projection of $S$ is infinite.
\end{enumerate}
\end{claim}
\begin{proof}
We first show that every $k$-cone $C\sub M^n$ can be written as a finite disjoint union of $k$-cones satisfying (1). Let $C=h(\cal J)$, with $\cal J$ and $h:V \subseteq M^{m+k}\to M^n$, as in Definition \ref{def-cone}. We work by induction on $m$. For $m=1$, the result obviously holds. Let $m>1$, and consider the set $T\sub S$ of all those elements whose at least two coordinates are the same. Without loss of generality, assume that for every $g\in T$, the first two coordinates are the same (otherwise the argument is similar). It is easy to see that $h\left(\bigcup_{g\in S\sm T} \{g\}\times J_g\right)$ has the right form,
and hence we may assume that $T=S$.
Now, for $g\in T$, let $g'$ denote the $(m-1)$-tuple obtained from $g$ by removing the first coordinate.
Let $$T'=\{g'\in M^{m-1}: g\in T\}$$
and $$V'=(T'\times M^k)\cap V,$$
and define $\cal J'=\{J'_{g'}\}_{g'\in T'}$, where $J'_{g'}=J_g$, and $h':V'\to M^n$ with $h'(g', t)=h(g, t)$. 
Then $h(\cal J)=h'(\cal J')$, with  $T'\sub M^{m-1}$. By inductive hypothesis, the result follows.

Now, we show that every $k$-cone that satisfies (1) can be written as a finite disjoint union of sets satisfying (1) and (2). Let $C=h(\cal J)$ be as above. We work again by induction on $m$. If $m=1$, the result obviously holds. Let $m>1$, and suppose that some coordinate projection of $S$ is not infinite, say the first, $\pi_1(S)=\{t_1, \dots, t_l\}$. Then $\cal J$ is the finite disjoint union of $\cal J_i=\{J_{ig}\}_{g\in S_{t_i}}$, $i=1, \dots, l$, where $J_{ig}=J_{(t_i, g)}$. Let $h_i: V_{t_i} \to M^n$
with $h_i(g, x)=h(t_i, g, x)$. Then each $h_i(\cal J_i)$ is still a $k$-cone satisfying (1), and $S_{t_i}\sub M^{m-1}$. By inductive hypothesis, the result follows.
\end{proof}




We will also  need the following lemma.

\begin{lemma}\label{embedcone} Every $k$-cone can be definably embedded into $M^{k+1}$.
\end{lemma}
\begin{proof}
Let $X=h(\cal J)\sub M^n$ be a $k$-cone, with $\cal J=\bigcup_{g\in S}\{g\}\times J_g \sub P^{m+k}$ and $h:V\to M^n$ as in Definition \ref{def-cone}. We first embed $S$ into $M$. By Fact \ref{fact-indt}, there are $\alpha_1, \dots, \alpha_m\in M$ which are $\dcl$-independent over $P$. Define $f:S\to M$ via
$$(x_1, \dots, x_m)\mapsto \alpha_1 x_1 +\dots + \alpha_2 x_m.$$
By choice of $\alpha_1, \dots, \alpha_m$, it follows that $f$ is injective.
Now, since $h$ is injective, we can embed $X$ into $M^{k+1}$ via $F:h(g, t)\mapsto (f(g),t).$
\end{proof}


\begin{cor}\label{embedX}
Let $X$ be a definable set of dimension $k$. Then there is a definable bijection $f:X\to M^n$ with $\dim cl(f(X))=k+1$.
\end{cor}
\begin{proof}
By cone decomposition, $X$ is a finite union of cones. By Lemma \ref{embedcone}, each of the cones can be definably embedded into $M^{k+1}$.  Then $X$ can  be definably embedded into finitely many disjoint copies of $M^{k+1}$, say, in $M^{k+2}$. The dimension of the closure of their union is still $k+1$.
\end{proof}

\begin{cor}\label{slfull2}
Let $X$ be the union of a full set and a set of smaller dimension. Then $X$ is in definable bijection with a strongly large set.
\end{cor}
\begin{proof} Suppose $\dim X=k$ and $X=Y\cup S$, with $Y$ full and $\dim S<k$. Since $\dim cl(Y)=k$, and using Corollary \ref{embedX} for $S$, we can easily embed $X$ into some $M^n$ via an $f$, such that $\dim cl( f(X))=k$. Therefore $f(X)$ is strongly large.
\end{proof}

Note that we only used that $Y$ is strongly large in the above proof.

\begin{proof}[Proof of Theorem \ref{indsl}] Let  $G=\la G, \cdot, 1\ra$ be   a definable group, with $G\sub M^n$ and $\ldim G=k$. For $a, b\in G$, we write $ab$ for $a\cdot b$. By Fact \ref{conedec} and Claim \ref{partitionG}, $G$ is a finite union of cones $C_1, \dots, C_p$, each of the form $h(\cal J)$, where $S=\pi(\cal J)$ satisfies Claim \ref{partitionG} (1) \& (2). Suppose towards a contradiction that $G$ is not in definable bijection with any strongly large set. We claim that some $k$-cone $C$ among $C_1, \dots, C_p$ must be of the form $h(\cal J)$, where, in addition, $S=\pi(\cal J)$ is infinite. Indeed, otherwise, $G$ would be the union of finitely many sets of the form $h(g, J_g)$ as in Remark \ref{hJfull}, together with a set of dimension $<k$. The former sets are all full, and, by  Lemma \ref{fullunion}, their union is also full. Hence $G$ is a union of a full set and a set of dimension $<k$, contradicting Corollary \ref{slfull2}.

Now fix a $k$-cone $C=h(\cal J)$ among $C_1, \dots, C_p$, with $S=\pi(\cal J)\sub P^m$ infinite and $m$ maximal such. 
By Claim \ref{partitionG}(1) \& (2), we can find two distinct elements $g_1, g_2\in S$ with all their $2m$ coordinates distinct.
Let $\Sigma$ be the set of all those $2m$ coordinates.
Let also $A\sub M$ be a finite parameter set that is used to define all cones $C_i=h_i(\cal J_i)$ and their associated functions $h_i$ and families of supercones $\cal J_i$.  Let $A_0\sub A$ be so that $A\sub \dcl(A_0 P)$ and $A_0$ is $\dcl$-independent over $P$.

$ $\\
\textbf{Case: $k=0$.} Since all of $C_1, \dots, C_p$ are $0$-cones, we may write $C_i=h_i(S_i)$, where $S_i\sub P^{k_i}$, for some $l_i\le m$. Let $i$ and $g_3\in S_i$ be so that
$$h(g_1) h(g_2) = h_i (g_3).$$
Since $|\Sigma|=2m$, there must be $a\in \Sigma \sm g_3$. Say $a\in g_2\sm g_1$ (if $a\in g_1\sm g_2$, the argument is symmetric). By injectivity of $h$,  $a\in \dcl(g_1, g_3, A_0, P_0)$, contradicting the fact that $A_0$ is $\dcl$-independent over $P$.


$ $\\
\textbf{Case: $k>0$.}  We need the following claim.

$ $\\
\noindent\textbf{Claim.} \emph{There are a $k$-cone  $D=h'(\cal J')$ among the $C_i$'s, with $\cal J'=\bigcup_{g\in S'} \{g\}\times J'_g$, a tuple $g_3\in S'$, and a  triple $(t_1, t_2, t_3)\in J_{g_1}\times J_{g_2}\times J_{g_3}$, such that $t_1 t_3 A_0 P$ is $\dcl$-independent, and}
$$ (*)\,\,\,\,\,\,\,
  h(g_1, t_1) h(g_2, t_2)= h'(g_3, t_3).$$
\begin{proof}[Proof of the claim]
Let $X=J_{g_1}\times J_{g_2}$. Then $X$ is a supercone in $M^{2k}$. By \textbf{(D6)}, there is a definable set $Z\sub X$ with $\ldim Z<2k$ and an $\cal L$-definable map $F:M^{2k}\to M^n$ such that the map $$(t_1, t_2)\mapsto h(g_1, t_1)h(g_2, t_2)$$ agrees with $F$ on $X':=(J_{g_1}\times J_{g_2})\sm Z$. By o-minimality, there is an open cell $U\sub cl(X)$, such that $F_{\res U}$ is continuous. By \cite[Lemma 4.16]{egh}, $U\cap X$ is a supercone in $M^{2k}$, and hence the set $T=(U\cap X)\sm Z$ also  has dimension $2k$. By Fact \ref{fact-indt}, there is $(t_1, t_2)\in T$, which is $\dcl$-independent over $A_0 P$.  Moreover, by \cite[Lemma 5.10]{egh}, there are a $k'$-cone $D$ among the $C_i$'s,  say $D=h'(\cal J')$, with $k'\le k$ and $\cal J'=\bigcup_{g\in S'} \{g\}\times J'_g$, and  $g_3\in S'$ such that $F(T)\sub h'(g_3, J'_{g_3})$. Let $t_3\in J'_{g_3}\sub M^{k'}$ be so that
$$h(g_1, t_1) h(g_2, t_2)=h'(g_3, t_3).$$
Since $t_2\in \dcl(t_1, t_3, A_0, P)$, it follows that $(t_1, t_3)$ is also $\dcl$-independent over $A_0 P$ and has dimension at least $2k$. Hence $k'=k$. 
\end{proof}
Let $D$ and $g_3$ be as in the claim.  By maximality of $m$, it must be that $g_3\in S'\sub P^l$, for some $l\le m$. Since all $2m$ coordinates of $g_1, g_2$ are distinct, there must be $a\in \Sigma\sm g_3$. Say $a\in g_2\sm g_1$.
By (*),
$$h(g_2, t_2)=h(g_1, t_1)^{-1} h'(g_3, t_3),$$
and hence
$$a\in g_2\sub \dcl(g_1, g_3, t_1, t_3, A_0, P_0),$$
contradicting  the fact that $t_1 t_2 A_0$ is $\dcl$-independent over $P$. 
\end{proof}

\section{A future direction}\label{sec-future}





 There are many tame expansions of o-minimal structures that support a nice notion of dimension, and hence where at least the methods of this paper could apply. 
Examples include real closed valued fields, closed ordered differential fields, expansions by a discrete set, expansions by a generic set, and $H$-structures. Here we point out a new direction which has not yet  been considered. For all relevant notions of NIP structures, the reader may consult \cite{simon-book} or \cite{nell}.  Assume that $\cal M$ is a distal structure, and let $\cal N$ be an expansion of $\cal M$, which is NIP, but not distal. This is the case, for example, with all pairs $\la \cal M, P\ra$  mentioned in Section \ref{sec-egh} (see \cite{hn}). Define the \emph{distal closure} operator $\dscl: \cal P(M)\rightarrow \cal P(M)$ as follows:
\[
a\in \dscl(A) \Lrarr \text{ $tp(a/A)$ is distal.}
\]
Work from \cite{nell} implies that for a dense pair of real closed field, a type $tp(a/A)$ is small (that is, it contains a small formula) if and only if $tp(a/A)$ is distal. Combined with work from \cite{egh}, we obtain that $\dscl$ in this setting is a pregeometry, and that the corresponding $\dscl$-dimension coincides with the large dimension (as in Section \ref{sec-egh}). The proposed direction is to explore  further expansions  \cal N where  $\dscl$ is a pregeometry, and, if $\cal M$ is o-minimal,  to check whether axioms \textbf{(D1)-(D6)} for the $\dscl$-dimension hold.





\end{document}